\newif\ifTwoColumn
\newif\ifTechReport

\pdfoutput=1

\TechReporttrue

\ifTwoColumn
\documentclass[letterpaper, 10 pt, conference]{ieeeconf}

\else
\ifTechReport
\documentclass[12pt,onecolumn]{article}
\usepackage[margin=2cm]{geometry}

\else
\documentclass[12pt,draftclsnofoot,onecolumn,letterpaper]{ieeeconf}

\fi
\fi

\ifTwoColumn
\fi

\usepackage{amsmath,amsthm,amssymb}
\usepackage{graphicx}
\usepackage{float}
\usepackage{bbm}
\usepackage{mathtools}
\PassOptionsToPackage{hyphens}{url}\usepackage{url}
\usepackage{cite}
\usepackage{epstopdf}
\usepackage{xcolor}

\allowdisplaybreaks
\newtheorem{theorem}{Theorem}
\newtheorem{lemma}[theorem]{Lemma}
\newtheorem{assumption}{Assumption}
\newtheorem{proposition}[theorem]{Proposition}
\newtheorem{definition}{Definition}
\newtheorem{remark}[theorem]{Remark}

\newcommand{\mbb}{\mathbb}
\renewcommand{\Re}{\mbb{R}}

\newcommand{\set}[2]{\left\{ #1\ \left| \ #2 \right. \right\}}
\newcommand{\norm}[1]{\lVert#1\rVert}

\usepackage[hidelinks]{hyperref}
\hypersetup{
	colorlinks=false 
}
\usepackage{flushend}

\title{\LARGE \bf
	A Concave Value Function Extension for the Dynamic Programming Approach to Revenue Management in Attended Home Delivery$^*$
}

\author{Denis Lebedev, Paul Goulart \& Kostas Margellos$^{1}$
	\thanks{*This work was kindly supported by SIA Food Union Management and by EPSRC, U.K., under Grant EP/P03277X/1.}
	\thanks{$^{1}$The authors are with the Department of Engineering Science, University of Oxford, Oxford, OX1 3PJ, United Kingdom. Email: {\tt\small \{denis.lebedev, paul.goulart, kostas.margellos\}@eng.ox.ac.uk}}%
}

\begin{document}
	\maketitle
	\thispagestyle{plain}
	\pagestyle{plain}	
	\begin{abstract}
		We study the approximate dynamic programming approach to revenue management in the context of attended home delivery. We draw on results from dynamic programming theory for Markov decision problems, convex optimisation and discrete convex analysis to show that the underlying dynamic programming operator has a unique fixed point. Moreover, we also show that -- under certain assumptions -- for all time steps in the dynamic program, the value function admits a continuous extension, which is a finite-valued, concave function of its state variables. This result opens the road for achieving scalable implementations of the proposed formulation, as it allows making informed choices of basis functions in an approximate dynamic programming context. We illustrate our findings using a simple numerical example and conclude with suggestions on how our results can be exploited in future work to obtain closer approximations of the value function.
	\end{abstract}
	\section{Introduction}\label{sec:intro}
	\subsection{Revenue Management and Attended Home Delivery}
	The expenditure of US households on online grocery shopping could reach \$100 billion in 2022 according to the Food Marketing Institute \cite{FMI2018}. Although growth forecasts vary and more conservative estimates lie, for example, at \$30 billion for the year 2021 \cite{PITCHBOOK2017}, the overall trend is clear: The online grocery sector is likely to grow if some of its main challenges can be overcome. 
	
	One of these challenges is managing the logistics as one of the main cost-drivers. In particular, one can seek to exploit the flexibility of customers by offering delivery options at different prices to create delivery schedules that can be executed in a cost-efficient manner. There are a number of ways to achieve this. Recent proposals include, for example, giving customers the choice between narrow delivery time windows for high prices and  \textit{vice versa} \cite{CAMPBELL2006} or charging customers different prices based on the area and their preferred delivery time \cite{ASDEMIR2009,YANG2016,YANG2017}. 
	
	In this paper, we focus on the latter. We refer to the problem of finding the profit-maximising delivery slot prices as the revenue management problem in attended home delivery, where ``attended'' refers to the requirement that customers need to be present upon delivery of the typically perishable goods, which is in contrast to, for example, standard mail delivery. Note that attended home delivery problems are more complex than standard delivery services, because goods need to be delivered in time windows that were pre-agreed with the customers. 
	
	We adopt a dynamic programming (DP) model of an expected profit-to-go function, the value function of the DP, given the current state of orders and time left for customers to book a delivery slot. This DP was initially devised in the fashion industry \cite{GALLEGO1994}, but subsequently adopted and refined by the transportation sector and the attended home delivery industry \cite{YANG2016}. 
	
	To find the optimal delivery slot prices, we need to compute the value function (at least approximately) for all states and times. The main challenge with this is that the state space of the DP grows exponentially with the set of delivery time slots, i.e.\ it suffers from the ``curse of dimensionality''. This means that for industry-sized problems, the value function cannot be computed exactly, even off-line, because there are too many states. Our ultimate objective is to compute improved value function approximations. Therefore, we study in this paper how the value function of the exact DP behaves mathematically in time and across state variables.
	
	In this paper we show that -- under certain assumptions -- for all time steps in the dynamic program, the value function admits a continuous extension, which is a finite-valued, concave function of its state variables. This result opens the road for achieving scalable implementations of the proposed formulation, as it becomes possible to make informed choices of basis functions in an approximate dynamic programming context.
	
	Improved value function approximations could finally be used for calculating optimal delivery slot prices. This has been shown by \cite{DONGETAL2009}, where it is proven that a unique set of optimal delivery slot prices exists, which can be found using simple Newton root search algorithms if estimates of the value function are known for all states and times.
	
	Our paper is structured as follows: In the remainder of Section \ref{sec:intro}, we introduce some notation. Then in Section \ref{sec:rm_problem_form}, we define the revenue management problem in attended home delivery and formulate it as a DP. In Section \ref{sec:theorem}, we state the definitions and assumptions that our analysis is based on and present our main result, Theorem \ref{th:v_conc}, which says that there exists a continuous extension of the value function of the exact DP that is a finite-valued, concave function in its state variables. Section \ref{sec:fixedpoint} contains reformulations of the DP into mathematically more convenient forms and develops a series of supporting results leading to the proof of Theorem \ref{th:v_conc}. Section \ref{sec:ex} presents a numerical illustration of the proposed scheme, while Section \ref{sec:conc} concludes the paper and suggests directions for future research.
	
	\subsection{Notation}\label{subsec:notation}
	Let	$\mathbf{1}$ denote a column vector of ones. Given some $(a,s)$, let $1_{a,s}$ be a column vector build by stacking the transposed rows of a unit matrix with $1$ at the $(a,s)$\textsuperscript{th} entry. Let $\Re_{+}$ be the non-negative real numbers and let dim$(\cdot)$ denote the dimension of its argument. Let conv$(\cdot)$ denote the convex hull of its argument.
	
	\section{Revenue Management Problem Formulation}\label{sec:rm_problem_form}
	In this section, we derive a discrete-state formulation of the revenue management problem in attended home delivery. 
	\subsection{Problem Statement}
	We model an online business that delivers goods to locations of known customers. We adopt a local approximation of the revenue management problem by dividing the service area geographically into a set of non-overlapping rectangular sub-areas \mbox{$A:=\{1,2,\dots,\bar{a}\}$}, where the customers in each area $a \in A$ are served by one delivery vehicle. This model resembles the setting in \cite{YANG2017}.
	
	We consider a finite booking horizon with possibly unequally-spaced time steps indexed by $t \in T:=\{1,2,\dots,\bar{t}\,\}$. We refer to \cite[Section 4.3]{YANG2016} for details on how to obtain a customer arrivals model using a Poisson process with \textit{time-invariant} event rate $\lambda \in (0,1)$ for all $t\in T$ from a Poisson process with homogeneous time steps, but \textit{time-varying} event rate. The probability that a customer arrives from sub-area $a$ is given by $\pi(a) \in [0,1]$ with $\sum_{a\in A}\pi(a)=1$. 
	
	Customers can choose from a number of (typically 1-hour wide) delivery time windows, which we call slots $s \in S$, where $S:=\{1,2,\dots,\bar{s}\}$. Let $s=0$ correspond to a customer not choosing any slot. Each sub-area/delivery slot pair $(a,s)$ is assigned a delivery charge $d_{a,s} \in \left[\underline{d},\bar{d}\,\right] \cup \infty $, for some minimum allowable charge $\underline{d} \in \mathbb{R}$ (which is typically, though not necessarily, positive) and some maximum allowable charge $\bar{d} \geq \underline{d}$. The role of $d_{a,s}=\infty$ is a convention to indicate that slot $s$ is not offered in area $a$. This will be explained in more detail when introducing the customer choice model below.
	
	We define \mbox{$d:=\set{d_{a,s}}{(a,s)\in A\times S}$} consisting of the set of delivery charges that the business decides to charge at any time step $t \in T$, where it is to be understood that $d$ is a stacked vector including the different values that $d_{a,s}$ can take as $a$ and $s$ vary. Let the set of allowed decision vectors be \mbox{$D:=\set{d}{d_{a,s} \in \left[\underline{d},\bar{d}\,\right] \text{ for all } (a,s) \in A\times S}$}. 
	
	For each sub-area/delivery slot pair \mbox{$(a,s) \in A\times S$}, we denote the number of placed orders by $x := \set{x_{a,s}}{(a,s)\in A\times S}\in \mathbb{Z}^{|A||S|}$. We define $X:=\set{x}{0\leq x_{a,s} \leq \bar{x}_{a,s} \text{ for all } (a,s)\in A\times S}$, where $\bar{x}_{a,s}$ is a scalar indicating the maximum number of deliveries that can be fulfilled in the sub-area/delivery slot pair $(a,s)$. In general, we do not require the maximum number of deliveries to be the same for all areas and all slots, e.g.\ because this will depend on the size of the delivery area. Examples of computing this quantity can be found in \cite[Section 4]{YANG2017}. Let us also define $\bar{x}:=\set{\bar{x}_{a,s}}{(a,s) \in A\times S}$. Let $r \in \Re$ denote the expected net revenue of an order, i.e.\ expected revenue minus costs prior to delivery. This is assumed to be invariant across all orders. We define 
	\begin{equation}\label{eq:cost}
	C(x):=
	\begin{dcases}
	C_{\mathbb{\Re_+}}(x) & x \in X \\
	\infty & \text{otherwise,}
	\end{dcases}
	\end{equation}
	where $C_{\Re_+}: X \to \Re_+$. The function $C$ approximates the delivery cost to fulfil the set of orders $x$. The precise delivery cost cannot be computed, as it is the solution to a vehicle routing problem with time windows, which is intractable for industry-sized applications \cite{TOTH2014}.
	
	Let the probability that a customer chooses sub-area/delivery slot pair $(a,s)$ if offered prices $d$ be $\Pi_{a,s}(d)$, such that $d \mapsto \Pi_{a,s}(d)\in[0,1)$ for all $(a,s) \in A\times S$. For all areas $a\in A$, note that $\sum_{s \in S}\Pi_{a,s}(d) = 1-\Pi_{a,0}(d)$, where $\Pi_{a,0}$ denotes the probability of a customer from sub-area $a$ leaving the online ordering platform without choosing any delivery slot. A typical choice for $\Pi_{a,s}$ is the multinomial logit model that was also used in \cite{YANG2017}:
	\begin{equation}\label{eq:mnl}
	\Pi_{a,s}(d):=\frac{\exp(\beta_c+\beta_s+\beta_dd_{a,s})}{\sum_{k\in S}\exp(\beta_c+\beta_k+\beta_dd_{a,k})+1},
	\end{equation}
	where $\beta_c \in \Re$ denotes a constant offset, $\beta_s \in \Re$ represents a measure of the popularity for all delivery slots and $\beta_d < 0$ is a parameter for the price sensitivity. Note that the no-purchase utility is normalised to zero, i.e.\ for the no-purchase ``slot'' $s=0$, we have \mbox{$\beta_c+\beta_0+\beta_d d_{a,0} = \beta_c+\beta_0 = 0$} and hence, the $1$ in the denominator of \eqref{eq:mnl} arises from $\exp(\beta_c+\beta_0)=1$.
	
	Note that our results on the fixed point computation do not depend on the particular form of the customer choice model. We only require that it is a probability distribution and in the limit as $d_{a,s} \to \infty$ for all $(a,s)\in A\times S$, we have that $\Pi_{a,s}(d)$ tends to zero with a higher than linear rate of convergence. This is important, because otherwise the expected profit-to-go will be unbounded.
	
	For convenience, let the probability that a customer arrives from sub-area $a$ and chooses slot $s$ given prices $d$ be denoted by $p_{a,s}(d) :=\lambda\pi(a)\Pi_{a,s}(d)$. We define $p(d):=\set{p_{a,s}(d)}{(a,s)\in A\times S}$ and $P:=\set{p(d)}{d\in D}$. Finally, it is to be understood that sums over $a$ and $s$ are always computed over their entire sets $A$ and $S$, respectively. Similarly, decisions $d$ are made such that $d \in D$. 
	\subsection{Dynamic Programming Formulation}
	We can express the problem described above as a DP. The expected profit-to-go, $V_t(x)$, closely resembles the DP formulation in \cite{YANG2017} and we define it as
	\ifTwoColumn
	\begin{equation}
	\begin{split}\label{eq:rm}
	V_{t-1}(x)&:=\underset{d}{\max}\left\{\sum_{a,s}p_{a,s}(d)\left[r+d_{a,s}+V_{t}(x+1_{a,s}) \right.\right.\\
	&\phantom{:=\;}\left.\left.-V_{t}(x)\right]\vphantom{\sum_{1}}+V_{t}(x) \right\}\text{ for all } x \in X, t\in T,\\
	&\phantom{:=\;}\text{ where } V_{\bar{t}+1}(x) = -C(x)\text{ for all } x \in X,
	\end{split}
	\end{equation}
	\else
	\begin{align}\label{eq:rm}
	V_t(x):=&\underset{d}{\max}\left\{\sum_{a,s}p_{a,s}(d)\left[r+d_{a,s}+V_{t+1}(x+1_{a,s})-V_{t+1}(x)\right]+V_{t+1}(x)\right\}\quad \forall x \in X, t\in T,\nonumber \\
	&\text{ where } V_{T+1}(x) = -C(x)\quad \forall x \in X,
	\end{align}
	\fi
	i.e.\ $C(\cdot)$ denotes the terminal condition. The difference $V_t(x)-V_t(x+1_{a,s})$ represents the value foregone by accepting a additional (discrete spatial) order, which in economic terms is the opportunity cost of an order. Note that -- similar to \cite{YANG2017} -- we ignore any vehicle load capacity constraints in the problem, as they are much less restricting than the time constraints on the delivery slots. Therefore, including the vehicle load capacity constraints would only increase computational costs, but would not substantially improve the decision policy. 
	
	For convenience in the sequel, let us define an abstract operator notation which expresses \eqref{eq:rm} in a more compact form:
	\begin{equation}\label{eq:abstract_operator}
	V_{t-1}:=\mathcal{T} V_{t}.
	\end{equation}
	\section{Concave Continuous Extension Theorem}\label{sec:theorem}
	As the state space $X$ is discrete, it is not possible to establish convexity properties from standard, i.e. continuous, convexity theory. In this section, we therefore first provide some definitions from discrete convex analysis and the assumptions upon which our main results are based. We then state our main result, Theorem \ref{th:v_conc}, and two intermediate results, Theorem \ref{pr:fixedpoint} and Proposition \ref{pr:ind_step}.
	\subsection{Definitions}
	\begin{definition}
		We define the set of stochastic vectors in $X$ as 
		\begin{equation}
		\mathcal{V}_X:=\set{v\in \Re_{+}^N}{\sum_{i\in N}v_i=1, N=\dim(X)}.
		\end{equation}
	\end{definition}
	\begin{definition}
		Let $x\in X$ and let $Q$ be a finite set. Then $Q$ is defined to be an enclosing set of $x$ if $x \in$ conv$(Q)$.
	\end{definition}
	\begin{definition} 
		We define $\mathcal{Q}(x)$ as the set of all sets $Q$ enclosing $x$.
	\end{definition}
	\begin{definition}[cf.\ {\cite[(2.1)]{MUROTA2001}}]\label{de:conc_clo}
		Let $a \in \Re^N$ and $b \in \Re$. Then the concave closure $\tilde{f}: \Re^N \to \Re \cup -\infty$ of a function \mbox{$f:\mathbb{Z}^N \to \Re \cup -\infty$} is defined as
		\begin{equation}
		\tilde{f}(x):=\inf\set{a^\intercal x + b}{a^\intercal y + b \geq f(y) }
		\end{equation}
		for all $y \in \mathbb{Z}^N$ and for all $x\in \Re^N$.
	\end{definition}
	\begin{definition}[{cf.\ \cite[Lemma 2.3]{MUROTA2001} and \cite[Proposition 2.31]{ROCKAFELLAR98}}]\label{de:conc_ext}
		A function $f: \mathbb{Z}^N \to \Re \cup -\infty$ is concave extensible if and only if any of the following equivalent conditions hold:
		\begin{itemize}
			\item[(a)] The evaluations of $f$ coincide with the evaluations of its concave closure $\tilde{f}$, i.e.\ $f(x)=\tilde{f}(x)$ for all $x\in \mathbb{Z}^N$.
			\item[(b)] For all $x\in X$ and for all $Q\in \mathcal{Q}(x)$, the evaluation of $f$ at $x$ does not lie below any possible linear interpolation of $f$ on the points $q\in Q$, i.e.\ for all $x\in X$, for all $Q\in \mathcal{Q}(x)$ and for all $\mu \in \mathcal{V}_X$, such that $ x = \sum_{q\in Q}\mu_q q$, it holds that
			\begin{equation} \label{eq:conc_ext}
			f(x) \geq \sum_{q\in Q} \mu_q f(q).
			\end{equation}
		\end{itemize}
	\end{definition}
	
	\subsection{Assumptions}
	We now state the assumptions that our main result builds on.
	
	\begin{assumption}\label{ass:c}
		The negative cost function $-C$ is concave extensible.
	\end{assumption}
	
	\begin{assumption}\label{ass:profit_pos}
		The marginal cost of an additional, feasible order is always smaller than the maximum marginal profit, i.e.\ $C(x+1_{a,s})-C(x)\leq \bar{d}+r$, for all \mbox{$ a \in A, s \in S, x \in X$}.
	\end{assumption}
	Let us define:\footnote{Section \ref{subsec:indstep_change} details why we have dropped the argument $d$ in $p(d)$ and $p_{a,s}(d)$ and write simply $p$ and $p_{a,s}$ instead.}
	\begin{equation}\label{eq:g_t}
		g_t(x,p):=\sum_{a,s}p_{a,s}\left[V_t(x+1_{a,s})-V_t(x)\right]+V_t(x)
	\end{equation}
	 for all $x\in X, p\in P$ and $t\in T$.
	\begin{assumption}\label{ass:g_t}
		For all $t\in T$, we assume that $g_t$ is concave extensible in $(x,p)$, for any $V_t$ that is concave extensible in $(x,p)$.
	\end{assumption}
	Assumption \ref{ass:c} is satisfied for the class of affine functions typically used in the literature \cite{YANG2017}. Assumption \ref{ass:profit_pos} is not restrictive, as it offers the means to ensure that every additional order can generate profit. Otherwise, the delivery slot prices, which maximise \eqref{eq:rm}, would always be $d_{a,s}=\infty$ for all $(a,s)\in A\times S$, resulting in not offering any slots. Assumption \ref{ass:g_t} intuitively states that $g_t(x,p)$, a weighted perturbation of $V_t$ in $x$, should be concave extensible. This assumption appears to be strong, but it can always be satisfied by choosing a small enough customer arrival probability $\lambda$. Hence, $p_{a,s}=\lambda\pi(a)\Pi_{a,s}$ and $q_{a,s}^{(p)}=\lambda\pi(a)\Pi_{a,s}^{(q)}$ can be made arbitrarily small and one of the following two cases occurs:
	
	1) Consider that $g_t(x,p)$ is \textit{strictly} concave extensible, by which we mean that the condition for concave extensibility \eqref{eq:conc_ext} is satisfied with strict inequality, i.e.\ $V_t(x) - \sum_{q\in Q}\mu_q V_t(q) = \epsilon_t(x,Q) > 0$. The inequality condition for concave extensibility \eqref{eq:conc_ext} of $g_t(x,p)$ then becomes
	\begin{equation}\label{eq:gt_cond}
		\begin{split}
		&\sum_{a,s}p_{a,s}\left[V_t(x+1_{a,s})-V_t(x)\right]+\epsilon_t(x,Q) \\
		\geq&\sum_{q\in Q}\left\{\sum_{a,s}q^{(p)}_{a,s} \left[V_t(q^{(x)}+1_{a,s})-V_t(q^{(x)})\right]+V_t(q^{(x)})\right\},
		\end{split} 
	\end{equation}
	where $(x,p) = \sum_{q\in Q}\mu_q \left(q^{(x)},q^{(p)}\right)$. Let us define $w_t:=\min \{V_t(x+1_{a,s})-V_t(x)\}$ as well as $W_t:=\max\{V_t(q^{(x)}+1_{a,s})-V_t(q^{(x)})\}$, where the minimisation is taken with respect to $a$ and $s$ and the maximisation is taken with respect to $q^{(x)}$, $a$ and $s$. Then the inequality in \eqref{eq:gt_cond} can be tightened to obtain 
	\begin{equation}\label{eq:tighten}
	\begin{split}
	\sum_{a,s}p_{a,s}w_t+\epsilon_t(x,Q) \geq& \sum_{q\in Q} \mu_q\sum_{a,s}q_{a,s}^{(p)}W_t \\
	\Longleftrightarrow \lambda\leq&\frac{\epsilon_t(x,Q)}{(W_t-w_t)\sum_{a,s}\pi(a)\Pi_{a,s}}.
	\end{split}
	\end{equation}
	As $\epsilon_t(x,Q) > 0, W_t \geq w_t$ and $\sum_{a,s}\pi(a)\Pi_{a,s} > 0$, there exists a $\lambda > 0$ that satisfies the above inequality for all $t \in T$. Therefore, $\lambda$ implicitly depends on $T$, but for simplicity, we will just write $\lambda$.
	
	2) Consider that $\epsilon_t(x,Q)=0$, i.e.\ the points lie on a hyperplane. Then $V_t(x+1_{a,s})-V_t(x)=V_t(q+1_{a,s})-V_t(q)$ for all $x\in X, q\in Q, a\in A$ and $s\in S$. Therefore, \eqref{eq:gt_cond} holds with equality, independently of the particular choice of $\lambda$.
	\begin{remark}
	There are a few more special cases, where Assumption \ref{ass:g_t} is trivially satisfied. For example, if $D$ is a degenerate interval, i.e.\ it only contains a single $d$, then there is only a single $p$, which also means that $q^{(p)}=p$ for all $q\in Q$ and \eqref{eq:gt_cond} holds for all $\lambda$.
		
	Similarly, in the case that there is only one delivery sub-area/slot pair $(a,s)$, i.e.\ $A$ and $S$ are both singleton sets, the inequality condition \eqref{eq:gt_cond} simplifies to
	\begin{equation}
	\begin{split}
	&\phantom{\geq \;\,\,}p\left[V_t(x+1)-V_t(x)\right]+V_t(x) \\
	&\geq\sum_{q\in Q} \mu_q \left\{p \left[V_t\left(q^{(x)}+1\right)-V_t\left(q^{(x)}\right)\right]+V_t\left(q^{(x)}\right)\vphantom{q^{(p)}}\right\}.
	\end{split}
	\end{equation}
	As in this scenario $V_t$ is a one-dimensional concave extensible function, we can express this inequality in terms of the concave closure of $V_t$.
			\begin{equation}\label{eq:1d}
			\tilde{V}_t(x+p)\geq \sum_{q\in Q}\mu_q\tilde{V}_t\left(q^{(x)}+q^{(p)}\right).
			\end{equation}
			Noting that $x+p=\sum_{q\in Q}\mu_q\left[q^{(x)}+p\right]$, \eqref{eq:1d} holds, since $\tilde{V}_t$ is concave (in the ordinary sense) by Definition \ref{de:conc_clo}.
		\end{remark}

	\subsection{Statement of Main Results}
	Based on the aforementioned definitions and assumptions, we formulate our main result:
	\begin{theorem}\label{th:v_conc}	
		Under Assumptions \ref{ass:c}, \ref{ass:profit_pos} and \ref{ass:g_t}, $V_t$ is finite-valued, concave extensible in $x$ for all $t\in T$.
	\end{theorem}
	
	The proof of Theorem \ref{th:v_conc} mainly depends on the following two results:
	\begin{theorem}\label{pr:fixedpoint}
		Under Assumption \ref{ass:profit_pos}, the unique fixed point of \eqref{eq:abstract_operator} is given by
		\begin{equation}
		\label{eq:dp_fixedpoint}
		V^*(x):= (\bar{d}+r)\mathbf{1}^\intercal (\bar{x}-x)-C(\bar{x})\text{ for all } x\in X.
		\end{equation}	
	\end{theorem}
	\begin{proposition}\label{pr:ind_step}
		Consider Assumption \ref{ass:g_t} and fix any $t\in T$. If $V_t$ is concave extensible in $x$, then $\mathcal{T}V_t$ is also concave extensible in $x$.
	\end{proposition}
	
	We now prove Theorem \ref{pr:fixedpoint}, Proposition \ref{pr:ind_step} and Theorem \ref{th:v_conc} in Section \ref{sec:fixedpoint}.
	\section{Fixed Point Theorem and Concavity Preservation
	}\label{sec:fixedpoint}
	\subsection{Fixed Point Characterisation}
	To prove Theorem \ref{pr:fixedpoint}, we first establish a helpful, alternative formulation of \eqref{eq:rm}. We then state some supporting lemmata and proceed with the proof. The proofs of all supporting lemmata can be found in the Appendix.
	
	\subsubsection{Stochastic Shortest Path Problem Reformulation}\label{subsec:fixedpoint_spp}
	In this section, we reformulate \eqref{eq:rm} as an equivalent stochastic shortest path problem, a special version of an undiscounted, finite-state, discrete-time Markov decision problem. We can follow the arguments of \cite[Chapter 3]{BERTSEKAS2012} to rewrite \eqref{eq:rm} as
	\begin{equation}\label{eq:ssp}
	V_{t-1}(x):=\underset{d}{\max} \sum_{y \in X}P_{x,y}(d)[\tilde{g}(x,d,y)+V_{t}(y)],
	\end{equation}
	where $P_{x,y}(d)$ and $\tilde{g}(x,d,y)$ are defined as follows: If \mbox{$y=x\in X$}, then $P_{x,y}(d):=1-\sum_{a,s}p_{a,s}(d)$ and for all $(a,s)\in A\times S$, if $y=x+1_{a,s}$, then $P_{x,y}(d)=p_{a,s}(d)$. In all other cases, $P_{x,y}(d)=0$. Note that the first case is associated with no transition, the second case is a valid, i.e.\ unit-sized, order and the third group covers all invalid cases.
	
	In a similar way, for all $(a,s)\in A\times S$, if $y=x+1_{a,s}$, then $\tilde{g}(x,d,y)=d_{a,s}+r$ and in all other cases, $\tilde{g}(x,d,y)=0$. For convenience, similarly to \cite[Chapter 3]{BERTSEKAS2012}, we define $g(x,d):= \sum_{y\in X}P_{x,y}(d) \tilde{g}(x,d,y)$ to simplify \eqref{eq:ssp}:
	\begin{equation}\label{eq:ssp_simple}	V_{t-1}(x)=\underset{d}{\max}\left\{g(x,d)+\sum_{y \in X}P_{x,y}(d)V_{t}(y)\right\}.
	\end{equation}
	
	Let $v^*(\cdot)$ be the solution to a different, stationary DP with the same transition probabilities $P_{x,y}(d)$ as in \eqref{eq:ssp_simple}, but $g(x,d):=1$ if $x \in X \backslash\{ \bar{x}\}$ and $g(x,d):=0$, otherwise.
	\begin{equation}\label{eq:dp_aux}
	v^*(x):=\begin{dcases}
	1+\underset{d}{\max}\left\{\sum_{y \in X}P_{x,y}(d)v^*(y)\right\} & x\in X \backslash \{\bar{x}\} \\
	1 & x=\bar{x}.
	\end{dcases}
	\end{equation}
	Note that $v^*(x)\geq1$ for all $x \in X$. Let us also define $\rho$ as
	\begin{align}
	\rho&:= \underset{x \in X} {\max}\frac{v^*(x)-1}{v^*(x)} \nonumber\\
	\phantom{\rho}&=\underset{x \in X} {\max}\left\{ \frac{\underset{d \in D}{\max}\sum_{y \in X} P_{x,y}(d)v^*(y)}{v^*(x)}\right\}.
	\end{align}
	Since we have $v^*(x) \geq 1$ for all $x \in X$, we conclude that $0 \leq \rho < 1$. Finally, let 
	\begin{equation}
	\norm{V_t}:=\underset{x \in X}{\max}\left\{\frac{|V_t(x)|}{v^*(x)}\right\}
	\end{equation} be a weighted sup-norm of $V_t$.
	\begin{lemma}\label{le:contract}
		The mapping defined by the operator $\mathcal{T}$ is contractive with modulus of contraction $\rho$, i.e.\
		\begin{equation}\label{eq:contraction}
		\norm{\mathcal{T}V_t-\mathcal{T}V_{t'}} \leq \rho\norm{V_t-V_{t'}}
		\end{equation}
		for all $t\in T$ and $t'\in T$.
	\end{lemma}
	Hence, we can prove Theorem \ref{pr:fixedpoint}, by showing that there exists a fixed point with the analytic expression in \eqref{eq:dp_fixedpoint}.
	
	\subsubsection{Proof of Theorem \ref{pr:fixedpoint}}\label{subsec:fixedpoint_proof}
	We start with the necessary and sufficient condition for $\mathcal{T}$ to have a fixed point $V^*$, which is $V^*=\mathcal{T}V^*$. This translates into \eqref{eq:rm} as
	\ifTwoColumn
	\begin{equation}\label{eq:dp_fixedpoint_cond}
	\begin{split}
	0&=\underset{d}{\max}\left\{\sum_{a,s}p_{a,s}(d)[r+d_{a,s}\right.\\
	&\phantom{=\;}\left.+V^*(x+1_{a,s})-V^*(x)]\vphantom{\sum_a}\right\}.
	\end{split}
	\end{equation}
	\else
	\begin{align}\label{eq:dp_fixedpoint_cond}
	0&=\underset{d}{\max}\left\{\sum_{a,s}p_{a,s}(d)[r+d_{a,s}+V^*(x+1_{a,s})-V^*(x)]\right\}.
	\end{align}
	\fi
	Substituting \eqref{eq:dp_fixedpoint} into \eqref{eq:dp_fixedpoint_cond} yields
	\begin{equation}
	\label{dp:fixedpoint_finalcond} 0= \underset{d}{\max} \left\{\sum_{a,s}p_{a,s}(d)[d_{a,s} -\bar{d}\,]\right\}.
	\end{equation}
	The values of all $p_{a,s}(d)$ are non-negative for all \mbox{$(a,s)\in A\times S$} and for all $d \in D$. The value of $[d_{a,s}-\bar{d}\,]$ is non-positive and $0$ only if $d_{a,s}=\bar{d}$ for all $(a,s) \in A\times S$. It follows that the maximum non-negatively weighted sum of the $[d_{a,s}-\bar{d}]$ terms is $0$, so \eqref{dp:fixedpoint_finalcond} holds. As the value function is invariant at $\bar{x}$ and at this point \mbox{$V^{*}(\bar{x})=V_{\bar{t}+1}(\bar{x})=-C(\bar{x})$}, $V^{*}$ is indeed a fixed point of $\mathcal{T}$ for all $x \in X$. \qed
	\subsection{Preserving Concavity}\label{sec:indstep}
	The structure of this section is similar to the structure of Section \ref{sec:fixedpoint}: We first reformulate \eqref{eq:rm}, then we state some supporting lemmata and finally, we prove Proposition \ref{pr:ind_step}.
	
	\subsubsection{Change of Decision Variables}\label{subsec:indstep_change}
	We start the derivation of Proposition \ref{pr:ind_step} by reformulating \eqref{eq:rm} as a maximisation over $p \in P$ instead of $d \in D$. As shown by \cite{DONGETAL2009}, this is possible, because the following unique mapping between $p$ and $d$ exists:
	\begin{equation}
	\frac{p_{a,s}}{p_{a,0}}=\exp(\beta_c+\beta_s+\beta_d d_{a,s}),
	\end{equation}
	where by solving with respect to $d_{a,s}$ we obtain
	\begin{equation} d_{a,s}=\beta_d^{-1}\left[\ln\left(\frac{p_{a,s}}{p_{a,0}}\right)-\beta_c-\beta_s\right].
	\end{equation}
	Hence, we can rewrite \eqref{eq:rm} and break it down into two parts:
	\begin{equation}
	\begin{split}
	&\begin{split}
	\mathcal{T}V_t(x)&= \underset{p}{\max}\sum_{{a,s}}p_{a,s}\left\{r+\beta_d^{-1}\left[\ln\left(\frac{p_{a,s}}{p_{a,0}}\right)-\beta_c-\beta_s\right]\right. \\
	&\phantom{=\;}+V_t(x+1_{{a,s}})-V_t(x)\vphantom{\sum_a}\left.\vphantom{\frac{a}{a}}\right\}+V_t(x)
	\end{split}\\
	&\phantom{\mathcal{T}V_t(x)}=\underset{p}{\max} \{f(p)+g_t(x,p)\},
	\end{split}
	\end{equation}
	where we have defined
	\begin{equation}\label{eq:f_def}
	f(p):= \sum_{{a,s}}p_{a,s}\left\{ r+\beta_d^{-1}\left[\ln\left(\frac{p_{a,s}}{p_{a,0}}\right)-\beta_c-\beta_{s}\right]\right\}
	\end{equation}
	and $g_t$ is from \eqref{eq:g_t}. This allows us to make the next statement:
	\begin{lemma}\label{le:f_conc}
		The function $f$ is concave extensible in $(x,p)$.
	\end{lemma}
	We are finally ready to prove Proposition \ref{pr:ind_step} showing that $\mathcal{T}$ preserves concave extensibility of $V_t$ in $x$. 
	
	\subsubsection{Proof of Proposition \ref{pr:ind_step} and Theorem \ref{th:v_conc}}\label{subsec:indstep_proof}
	By Lemmata \ref{le:f_conc} and Assumption \ref{ass:g_t}, $f$ and $g$ have continuous extensions $\tilde{f}$ and $\tilde{g}$, which are both jointly concave in $(x,p)$. Therefore, $\tilde{h}(x,p):=\tilde{f}(p)+\tilde{g}(x,p)$ is also jointly concave in $(x,p)$. We define $U(x):=\underset{p}{\max} \, \tilde{h}(x,p)$. This allows us to exploit a standard convex optimisation result (see \cite[Proposition 2.22]{ROCKAFELLAR98} or \cite[Section 3.2.5]{BOYD2004}), according to which the maximisation with respect to some variables of a continuous multivariate function that is jointly concave in all its variables, yields a concave function. Therefore $U$ is a concave function of $x$.
	
	Repeating the same calculation, now with the discrete $h(x,p):=f(x)+g(x,p)$ in place of $\tilde{h}(x,p)$, i.e.\ \mbox{$\mathcal{T}V_t(x)=\underset{p}{\max} \, h(x,p)$}, note that $h(x,p)=\tilde{h}(x,p)$ for all grid points $x\in X$. Therefore, $\mathcal{T}V_t(x)=U(x)$ for all grid points $x\in X$. This shows that $U$ is a continuous extension of $\mathcal{T}V_t$, which is concave in $x$. Hence, $\mathcal{T}V_t$ is concave extensible in $x$. \qed

	Having proved Theorem \ref{pr:fixedpoint} and Proposition \ref{pr:ind_step} in the previous sections, the stage is set for the proof of Theorem \ref{th:v_conc}. 
	
	By the definition of $C$, $V_{\bar{t}+1}$ is finite-valued and by Theorem \ref{pr:fixedpoint}, $V^{*}$ is also finite-valued. Hence for all $x \in X$, the difference $V^*(x)-V_{\bar{t}+1}(x)$ is finite-valued. Let us use Lemma \ref{le:contract}, according to which for all $t\in T$ and $t'\in T$:
	\begin{equation}
	\norm{\mathcal{T}V_t-\mathcal{T}V_{t'}} \leq \rho\norm{V_t-V_{t'}}.
	\end{equation}
	Let $V_t=V^*=\lim\limits_{t' \to -\infty}V_{t'}$, where the limit uniquely exists and is well-defined, let $V_{t'}=V_{\bar{t}+1}$ and apply $\mathcal{T}$ $N$ times:
	\begin{align}
	\norm{\mathcal{T}^NV^*-\mathcal{T}^NV_{\bar{t}+1}} &\leq \rho^N\norm{V^*-V_{\bar{t}+1}} \nonumber \\
	\norm{V^*-V_{\bar{t}+1-N}} &\leq \rho^N\norm{V^*-V_{\bar{t}+1}}.
	\end{align}
	As $\rho < 1$, the pointwise difference in $x$ between $V^*$ and $V_t$ is finite for all $t \in T$, which implies that $V_t$ is finite for all $x \in X, t\in T$.
	
	By Assumption \ref{ass:c}, $V_{\bar{t}+1}( x):=-C(x)$ for all $x\in X$ is concave extensible. Hence, due to Proposition \ref{pr:ind_step}, which is effectively an inductive step, we can conclude that for all $t \in T$, $V_t$ is finite-valued, concave extensible in $x$.\qed
	
	Due to this result and based on \cite[Chapter 3]{BERTSEKAS2012}, we can also show that the fixed point \eqref{eq:dp_fixedpoint} is unique. Assume by contradiction that there are two fixed points of \eqref{eq:abstract_operator}, $V^{*}$ and $V^{**}$. Substituting $V^*$ and $V^{**}$ for $V_t$ and $V_{t'}$ in \eqref{eq:contraction}, respectively we obtain
	\begin{equation}
	\norm{\mathcal{T}V^{*}-\mathcal{T}V^{**}} \leq \rho \norm{V^{*}-V^{**}}.
	\end{equation} 
	Applying $\mathcal{T}$ $N$ times and taking the limit as $N \to \infty$, yields
	\begin{equation}
	\begin{split}
	\lim\limits_{N \to \infty}\norm{(\mathcal{T})^NV^{*}-(\mathcal{T})^NV^{**}} &\leq \lim\limits_{N \to \infty}\rho^N \norm{V^{*}-V^{**}} \\
	\Longleftrightarrow \norm{V^{*}-V^{**}} &\leq \lim\limits_{N \to \infty}\rho^N \norm{V^{*}-V^{**}} \\
	\implies  \norm{V^{*}-V^{**}} &\leq 0.
	\end{split}
	\end{equation}
	By reversing the roles of $V^{*}$ and $V^{**}$, we can conclude that $V^{*}=V^{**}$ and therefore, there can be at most one fixed point.
	\section{Illustrative Example}\label{sec:ex}
	We illustrate our findings using a simple numerical example of a 1-area, 2-slot problem. The parameters are listed in Table \ref{tab:parameters} below.
	
	\begin{table}[h]
		\caption{The parameters of the numerical example.}
		\centering
		\renewcommand{\arraystretch}{1}
		\begin{tabular}{r|l}\label{tab:parameters}
			$\lambda$ & $0.5$ \\
			$\bar{t}$ & $200$ \\
			$\left[\underline{d},\bar{d}\,\right]$ & $[0,2]$ \\
			$r$ & 1 \\
			$S$ & $\{1,2\}$ \\
			$\bar{x}$ & $[4,4]$ \\
			$\beta_c$ & $1$ \\
			$\beta_d$ & $-1$ \\
			$\beta_{1}$ & $1$ \\
			$\beta_{2}$ & $-1$ \\
			$C_{\Re_+}(x)$ & $2+x_1+2x_2$	
		\end{tabular}
	\end{table}
	These parameters yield the terminal condition 
	\begin{equation}
	V_{\bar{t}+1}(x):=-C(x)=-2-x_1-2x_2
	\end{equation} and the fixed point 
	\begin{equation}
	\begin{split}
	V^*(x)&:=\underset{t \to -\infty}{\lim}V_t(x)\\
	&\phantom{:}=(\bar{d}+r)\mathbf{1}^\intercal (\bar{x}-x)-C(\bar{x})\\
	&\phantom{:}=10-3(x_1+x_2)
	\end{split}
	\end{equation} 
	for all $x\in X$. To illustrate the contractive mapping property described by \eqref{eq:contraction}, we compute $\rho$ numerically by solving the auxiliary DP in \eqref{eq:dp_aux} and compare it with the running contractive ratio with respect to the fixed point $V^*$:
	\begin{equation}\label{eq:contractive_ratio}
	\rho_t:=\frac{\norm{V^*-V_{t}}}{\norm{V^*-V_{t+1}}}, \text{ for all } t\in T.
	\end{equation}
	By choosing $V^*$ to be the reference point, we also show that the DP converges to the fixed point as defined in \eqref{eq:dp_fixedpoint}. Fig.\ \ref{fig:rho} below illustrates this by showing that $\rho$ is an upper bound for $\rho_t$ for all $t\in T$.
	\begin{figure}[h!]
		\centering
		\includegraphics[width=0.5\columnwidth,trim=0cm 0cm 0cm 0.5cm,clip]{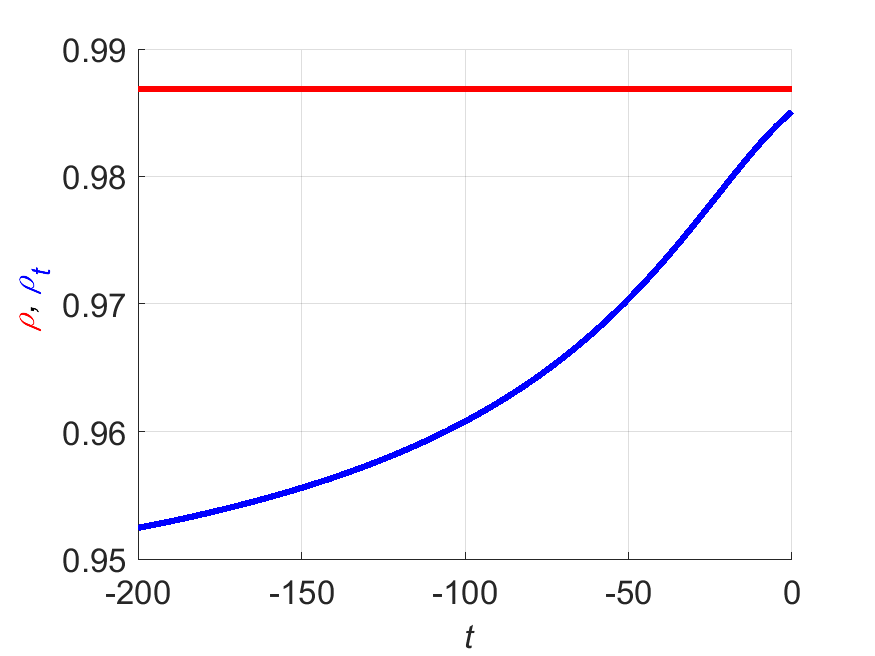}
		\caption{Modulus of contraction $\rho$ (red) and running contractive ratio $\rho_t$ (blue) at every time step in the booking horizon.}
		\label{fig:rho}
	\end{figure}
	
	We define a measure of discrete concavity
	\begin{equation}
	\epsilon(t):=\underset{x,Q\in \mathcal{Q}(x)}{\min}\left\{V_t(x)-\sum_{q\in Q} \mu_q V_t(q)\right\},
	\end{equation}
	such that $\mu \in \mathcal{V}, \sum_{q\in Q} \mu_q V_t(q)=V_t(x)$ for all $t\in T$. Note that $\epsilon \geq 0$ implies that $V_t$ is concave extensible. We compute this quantity by enumeration of all possible enclosing sets and plot the result in Fig.\ \ref{fig:gap}, from which it can easily be seen that $\epsilon_t(x)\geq 0$ for all $t\in T$.
	\begin{figure}[H]
		\centering
		\includegraphics[width=0.5\columnwidth,trim=0cm 0cm 0cm 0.5cm,clip]{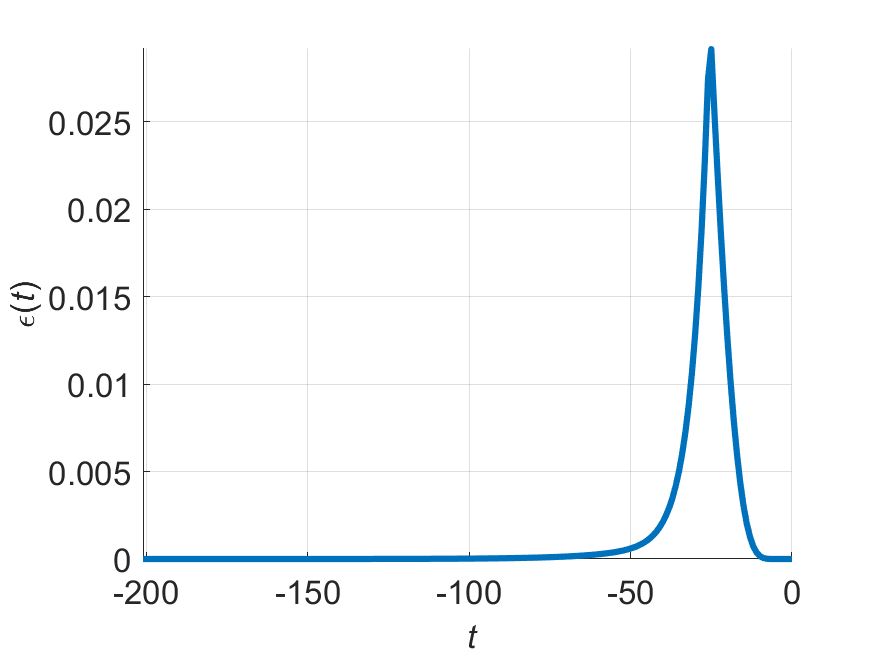}
		\caption{Concave extensibility measure for all time steps in the booking horizon.}
		\label{fig:gap}
	\end{figure}
	
	Finally, we plot the value function $V_t$ for $t=\bar{t}-10$ in Fig.\ \ref{fig:value_func} below. Note that the value function lies between the terminal condition and the fixed point. When it comes to approximating $V_t$, this information can be used to limit the range of basis function parameters, such that the approximated version of $V_t$ always lies between the terminal condition and the fixed point. 
	\begin{figure}[H]
		\centering
		\includegraphics[width=0.45\columnwidth,trim=5cm 7.7cm 4.5cm 7.9cm,clip]{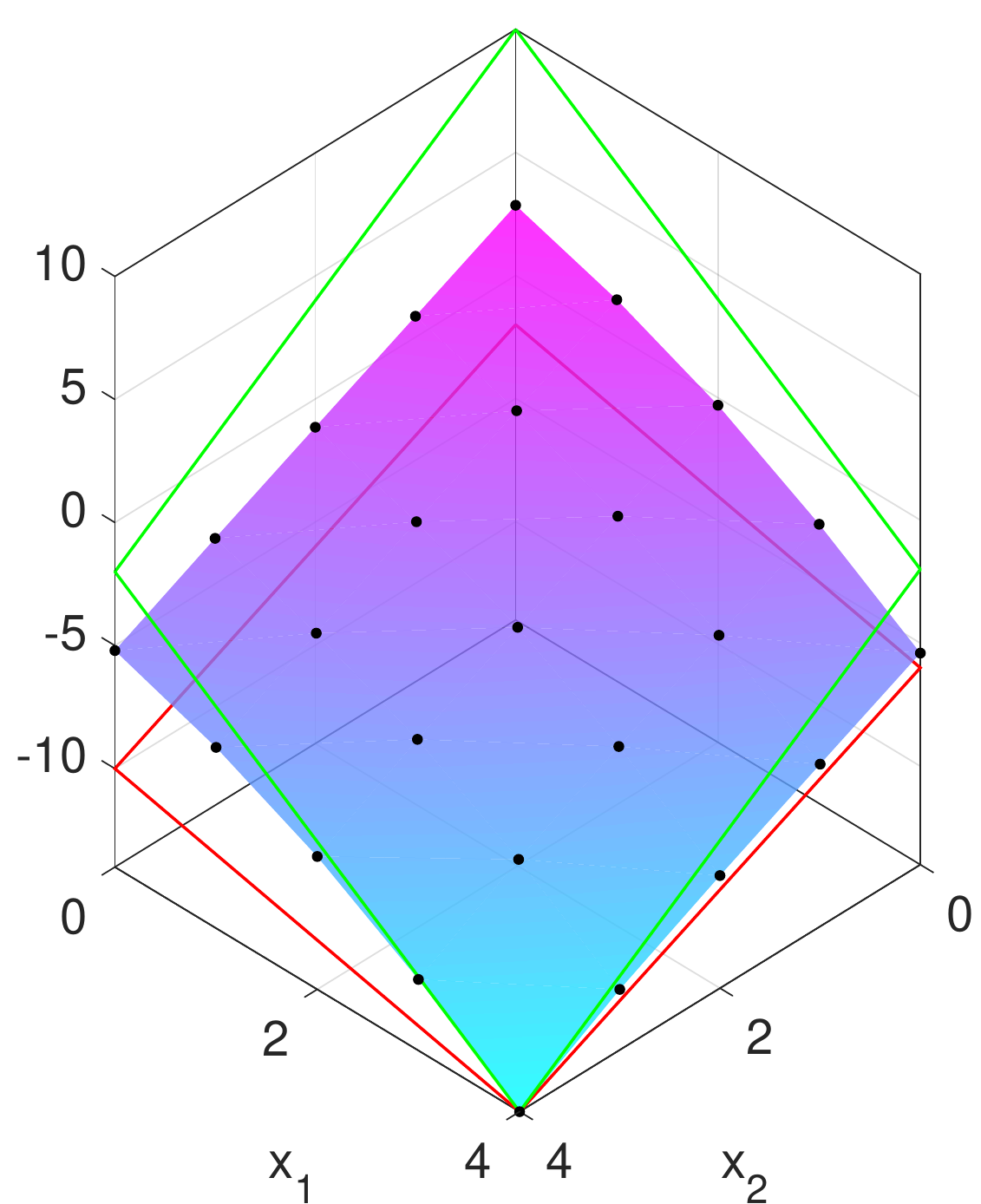}
		\caption{The value function at the terminal condition $t=\bar{t}+1$ (red), at the fixed point $t=-\infty$ (green) and at $t=\bar{t}-10$ (blue/violet colour gradient).}
		\label{fig:value_func}
	\end{figure}
	\section{Conclusions and Future Work}\label{sec:conc}
	
	We have studied the mathematical properties of the value function of a dynamic program modelling the revenue management problem in attended home delivery exactly. We have shown that the recursive dynamic programming mapping has a unique, finite-valued fixed point and concavity-preserving properties. Hence, we have derived our main result stating that -- under certain assumptions -- for all time steps in the dynamic program, the value function admits a continuous extension, which is a finite-valued, concave function of its state variables. We have illustrated our findings using a simple numerical example and now conclude with suggestions on how our results can be exploited in the future to obtain closer approximations of the value function.
	
	Recent approaches have estimated $V_t$ as an affine function of $x$ for each $t \in T$ \cite{YANG2017}. Based on our result, we believe that closer approximations can be found by pursuing different approximation strategies.
	
	One possible direction of future research involves investigating the use of parametric models comprising concave basis functions. This idea can be exploited directly by using the given DP formulation -- as suggested in \cite[Section 8.2]{POWELL2007} -- or by reformulating the problem as a linear program -- as shown by \cite{DEFARIAS2003}. Note that \textit{a priori} knowledge of concave extensibility of $V_t$ for all $t\in T$ creates some intuitive regularity. Therefore, it can be expected to get good approximations of $V_t$ from a relatively small sample size even with simple models.
	
	Another possible direction would be to adapt techniques that fit convex functions (or equivalently concave functions for our purposes) to multidimensional data. For example, \cite{KIM2004} and \cite{MAGNANI2009} show how data can be fitted by a function defined as the maximum of a finite number of affine functions. More sophisticated examples of convex (concave) function fitting techniques include adaptive partitioning \cite{HANNAH2013} and Bayesian non-parametric regression \cite{HANNAH2011}.
	
	Future work will show which strategy will find the best compromise between accuracy of the approximation and computational cost.
	
	\section*{Appendix}
	\begin{proof}[Proof of Lemma \ref{le:contract}]
		This proof is based on \cite[Chapters 1 and 3]{BERTSEKAS2012}. For any $t\in T$ and $t'\in T$ we have:
		\ifTwoColumn
		\begin{align}
		&\phantom{=\;\,}\|\mathcal{T}V_{t}-\mathcal{T}V_{t'}\| \nonumber\\
		&=\underset{x \in X}{\max}\frac{1}{v^*(x)}\left| \underset{d}{\max}\left\{g(x,d)+\sum_{y \in X}P_{x,y}(d)V_{t}(y)\right\}\right. \nonumber\\
		&\phantom{=\;}\left.-\underset{d'}{\max}\left\{g(x,d')+\sum_{y \in X}P_{x,y}(d')V_{t'}(y)\right\}\right| \nonumber
		\end{align}
		\begin{align}
		&\leq\underset{x \in X}{\max}\frac{\underset{d}{\max}\left|\sum_{y \in X}P_{x,y}(d)(V_{t}(y)-V_{t'}(y))\right|}{v^*(x)} \nonumber\\
		&=\underset{x \in X}{\max}\frac{\underset{d}{\max}\left|\sum_{y \in X}P_{x,y}(d)v^*(y)\frac{V_{t}(y)-V_{t'}(y)}{v^*(y)}\right|}{v^*(x)}\nonumber\\
		&\leq\underset{x \in X}{\max}\frac{\underset{d}{\max}\left|\sum_{y \in X}P_{x,y}(d)v^*(y)\right|}{v^*(x)}\|V_{t}-V_{t'}\| \nonumber\\
		&= \rho\|V_{t}-V_{t'}\|.
		\end{align}
		\else
		\begin{align}
		&\|\mathcal{T}V_{t}-\mathcal{T}V_{t'}\| \\
		=&\underset{x \in X}{\max}\frac{1}{v(x)}\left| \underset{d}{\max}\left\{g(x,d)+\sum_{y \in X}P_{x,y}(d)V_{t}(y)\right\}\right.\left.-\underset{d'}{\max}\left\{g(x,d')+\sum_{y \in X}P_{x,y}(d')V_{t'}(y)\right\}\right| \\
		\leq&\underset{x \in X}{\max}\frac{\underset{d}{\max}\left|\sum_{y \in X}P_{x,y}(d)(V_{t}(y)-V_{t'}(y))\right|}{v(x)} \\
		=&\underset{x \in X}{\max}\frac{\underset{d}{\max}\left|\sum_{y \in X}P_{x,y}(d)v^*(y)\frac{V_{t}(y)-V_{t'}(y)}{v^*(y)}\right|}{v(x)} \\
		\leq&\underset{x \in X}{\max}\frac{\underset{d}{\max}\left|\sum_{y \in X}P_{x,y}(d)v^*(y)\right|}{v(x)}\|V_{t}-V_{t'}\| \\
		=& \rho\|V_{t}-V_{t'}\|.
		\end{align}
		\fi
	\end{proof}
	\begin{proof}[Proof of Lemma \ref{le:f_conc}]
		By inspection, $f$ is only a function of the continuously-valued variable $p$, so it will be concave extensible in $(x,p)$ if it is concave in $p$ (in the ordinary, continuous sense of concavity). First, note that the function is separable, i.e.\
		\begin{equation}
		f(p)=\sum_a f_a(p),
		\end{equation}
		where we have defined
		\begin{equation}
		f_a(p):=\sum_s p_{a,s}\left\{ r+\beta_d^{-1}\left[\ln\left(\frac{p_{a,s}}{p_{a,0}}\right)-\beta_c-\beta_s\right]\right\}.
		\end{equation}
		Previously, \cite{DONGETAL2009} have shown that a structurally similar function is concave in its variables. We adopt their approach -- computing the Hessian and showing that it is negative definite -- to verify that $f_a$ is jointly concave in $\{p_{a,i}\}$ for all $i \in S\cup 0$. We first compute the first-order partial derivatives of $f_a$:
		\begin{equation}
		\begin{split}
		\frac{\partial f_a}{\partial p_{a,i}}&=[r+\beta_d^{-1}(\ln(p_{a,i}/p_{a,0})-\beta_c-\beta_i)]-\beta_d^{-1} \, \text{ for all } i\in S,\\
		\frac{\partial f_a}{\partial p_{a,0}}&=\sum_{s} -p_{a,s}\beta_d^{-1}p_{a,0}^{-1}, \\
		\frac{\partial f_a}{\partial p_{b,i}}&= 0 \ \text{ for all } a\neq b.
		\end{split}
		\end{equation}
		The second-order partial derivatives are:
		\begin{equation}
		\begin{split}
		\frac{\partial^2 f_a}{\partial p_{a,i}^2}&=\beta_d^{-1}p_{a,i}^{-1} \, \text{ for all } i \in S,\\
		\frac{\partial^2 f_a}{\partial p_{a,0}^2}&=\sum_{s}p_{a,s}\beta_d^{-1}p_{a,0}^{-2} ,\\
		\frac{\partial^2 f_a}{\partial p_{a,i} \partial p_{a,0}}&=-\beta_d^{-1}p_{a,0}^{-1} \, \text{ for all } i\neq 0,\\
		\frac{\partial^2 f_a}{\partial p_{a,i} \partial p_{a,j}} &= 0 \, \text{ for all } i\neq j,\\
		\frac{\partial^2 f_a}{\partial p_{b,i} \partial p_{a,j}} &= 0 \, \text{ for all } a\neq b.
		\end{split}
		\end{equation}
		Note that the partial derivatives show that $f_a$ is independent of all $\{p_{b,s}\}$, for which $a\neq b$. The resulting Hessian $H_a$ of $f_a$ with its second partial derivatives with respect to $\{p_{a,i}\}$ for all $i \in S\cup 0$ is:
		\begin{equation}
		\begin{split}
		H_a&=\beta_d^{-1}\begin{bmatrix}
		p_{a,1}^{-1} & \dots & 0 & -p_{a,0}^{-1} \\
		\vdots & \ddots &\vdots & \vdots \\
		0 & \dots & p_{a,\bar{s}}^{-1} & -p_{a,0}^{-1} \\
		-p_{a,0}^{-1} & \dots & -p_{a,0}^{-1} & p_{a,0}^{-2}\sum_{s}p_{a,s}
		\end{bmatrix} \\
		&=\begin{bmatrix}
		A & B \\
		B^\intercal & C
		\end{bmatrix},
		\end{split}
		\end{equation}
		where we have defined block sub-matrices $A,B,B^\intercal$ and $C$ such that $C$ is a scalar corresponding to the last entry of $H_a$. Note that $A$ is negative definite, because $p_{a,s}\geq0$ for all $a\in A, s\in S \cup 0$ and $\beta_d^{-1}<0$. We compute the Schur complement of $A$ in $H$:
		\begin{equation}
		H/A=\beta_d^{-1}\left(\sum_{s}p_{a,s} p_{a,0}^{-2}-p_{a,0}^{-2}\sum_{s}p_{a,s}\right)=0.
		\end{equation}
		As $A$ is negative definite and as $H/A$ is non-positive, $H_a$ is negative semi-definite. This implies that for all $a$, $f_a$ is concave in $p$. As taking the sum of concave functions preserves concavity (see \cite[Section 3.2.1]{BOYD2004}), $f$ is also concave in $p$. Hence, $f$ is concave extensible in $p$.
	\end{proof}
	\bibliographystyle{IEEEtran}
	\bibliography{IEEEabrv,AConcaveValueFunExt}
\end{document}